\documentclass[graybox]{svmult}

\usepackage{mathptmx}       
\usepackage{helvet}         
\usepackage{courier}        
\usepackage{type1cm}        
\usepackage{makeidx}         
\usepackage{graphicx}        
\usepackage{multicol}        
\usepackage[bottom]{footmisc}

\usepackage{amsmath} 

\makeindex             

\usepackage{amsfonts}
\newcommand{\N}{{\mathbb N}}
\newcommand{\Z}{{\mathbb Z}}
\newcommand{\R}{{\mathbb R}}
\newcommand{\C}{{\mathbb C}}

\newcommand{\EE}{B} 

\newcommand{\be}{\begin{equation}}
\newcommand{\ee}{\end{equation}}

\newcommand{\ben}{\begin{enumerate}}
\newcommand{\een}{\end{enumerate}}
\newcommand{\bit}{\begin{itemize}}
\newcommand{\eit}{\end{itemize}}
\newcommand{\edoc}{\end{document}}

\begin{document}

\title*{On the completeness of trajectories for some mechanical systems}
\titlerunning{Completeness of trajectories}
\author{Miguel S\'anchez}
\institute{Miguel S\'anchez  \at Departamento de Geometr\'{\i}a y
Topolog\'{\i}a, Universidad de Granada. Facultad de Ciencias,
Campus Fuentenueva s/n. E18071 Granada
(Spain), \email{sanchezm@ugr.es}. \\
Partially supported by Spanish grants with Feder funds P09-FQM-4496 (J. Andaluc\'{\i}a) and MTM2010--18099 (Mineco). 
}
%
%
\maketitle

\abstract{ The classical tools which ensure the completeness of
both, vector fields and second order differential equations for
mechanical systems, are revisited. Possible extensions  in three directions are
discussed: infinite dimensional Banach (and
Hilbert) manifolds, Finsler metrics and pseudo-Riemannian spaces,
the latter including links with some relativistic spacetimes. Special
emphasis is taken in the cleaning up of known techniques, the
statement of open questions and the exploration of prospective
frameworks. }



\section{Introduction}

As explained in the classical Abraham \& Marsden book \cite[p.
71]{AM}, the completeness of vector fields {\em is often stressed
in the literature since it corresponds to well-defined dynamics
persisting eternally}. However, in many circumstances one has to
live with incompleteness and, in this case, incompleteness may
mean  the failure of our model. Remarkably, this happens in
General Relativity, where singularities have become so common
(Schwarzschild spacetime, Raychauduri equation, theorems by
Penrose and Hawking...) that one expects to find incompleteness
under  physically reasonable general assumptions
---and one hopes that the quantum viewpoint will be able to
explain the physical meaning of singularities. In any case, the
possible completeness or incompleteness  becomes a fundamental
property of the model.

In his early works at the beginning of the seventies, Marsden
gave  two remarkable results on completeness. The first, in
collaboration with Weinstein \cite{WM},  extends  previous works
on the completeness of Hamiltonian vector fields by Gordon
\cite{Go},  Ebin \cite{Eb} and others. The second, about the
geodesic completeness of compact homogeneous pseudo-Riemannian
manifolds \cite{MaHomog}, was one of the few results ensuring
completeness instead of incompleteness in the Lorentzian setting
of that time. The results on the side of geodesics in the
Lorentzian setting have increased notably since then (see the
review \cite{CS}). Moreover, some connections with the original
Riemannian results for Hamiltonian systems have appeared. This has
been a stimulus for the recent update and extension of such
Riemannian results carried out by the author and his coworkers in
\cite{CRS}.

The aim of this paper is to revisit these results,
formulating them in a general framework, and pointing out new open
questions as well as new lines of study. The paper is organized into
three parts. In the first  (Section \ref{s1}), some
preliminaries on infinite-dimensional Banach manifolds endowed
with Finsler metrics are introduced. From our viewpoint, this is
the  natural framework for the completeness of first order
systems (vector fields),  and some second order ones can be
reduced to this setting.

In Section \ref{s2} we study completeness for both, first and
second order systems. For first order, we  review some old results
\cite{Go, Eb, WM, AM, AMR}  formulating them in the general Banach
Finsler case, and also allowing  the time-dependence of the vector
fields. We introduce {\em primary bounds}
  (Definition \ref{dprimarycomplete}) here, which allow the purification of
previous techniques (Theorem \ref{tsublinear}). For second order,
i.e., trajectories accelerated by
 potentials and other time-dependent forces, we give a general
result on completeness in Riemannian Hilbert manifolds (Theorem
\ref{t1}), which summarizes and extends those in \cite{Go, Eb, WM,
CRS}. The latter are also simplified technically because, even
though our proof uses  comparison criteria between differential
equations as in previous references,  here such criteria are
reduced essentially to the elementary lemma~\ref{lsubsol}
---and the bounds through  {\em positively complete} functions introduced
in \cite{WM} reduce to primary bounds as well. We suggest the
possibility of going further in two directions: the time-dependence
of the potentials and the Finsler Banach framework.

Section \ref{s4} deals with (finite-dimensional) pseudo-Riemannian
manifolds. Here there is a great diversity of results and techniques
(see \cite{CS}), and we focus on two topics. Firstly, results regarding
manifolds with a high degree of symmetry. In particular, the
extension of Marsden's Theorem \ref{tMa} to conformally related
metrics (Theorem \ref{tRS}), is explained by using the techniques
in the previous section. Secondly, the geometry of wave type
spacetimes. This provides a simple link between Riemannian and
Lorentzian results (Theorem \ref{twaves}) with new exciting open
questions
---some of them collected together at the end.

\section{Preliminaries on infinite-dimensional manifolds}
\label{s1}

Some preliminaries on  Banach  manifolds are introduced here.
Results on the elements which will be relevant for the posterior
results will be gathered together, and a framework for tentative
generalizations will be provided. Special emphasis is focused on
the role of paracompactness for the ambient manifold, as this
condition will be equivalent to the existence of a $C^0$-Finsler
metric such that its associated distance metrizes the manifold
topology. The role of smoothability for Finsler metrics is also
emphasized. Essentially, $C^0$ smoothability is sufficient for
distance estimates in first order problems (Section \ref{s2.1}),
but further smoothability may be required for the development of
second order ones (Section \ref{s2.3}).

 We will follow conventions
on Banach and Hilbert manifolds as in the original papers by
Palais \cite{PalaisProcSym_CriticalPointTh68, PalaisTop_LustSch66,
PalaisTop_Hilbert63}, as well as books such as Abraham, Marsden \&
Ratiu \cite{AMR}, Lang \cite{Lang}, Deimling \cite{De}, Kriegl \&
Michor \cite{Michor} or Moore's notes \cite{Moore}.

\runinhead{Topological conventions on Banach manifolds.} Any
Banach manifold $M$ will be always assumed $C^k$ with $k\geq 1$,
as well as {\em connected, Hausdorff and paracompact} and, thus,
normal\footnote{In particular, our Banach manifolds will be always
regular and, so, some difficulties pointed out by Palais in
\cite{PalaisProcSym_CriticalPointTh68} (see Sect. 2 including the
Appendix therein), will not apply. The central role of
paracompactness from the topological viewpoint is stressed in
Figure \ref{fig1}. Notice that, as a difference with the finite
dimensional case, second countability does not imply
paracompactness (see for example \cite{MargalefOuterelo},
\cite[Sect. 27.6]{Michor} or
\cite{PalaisProcSym_CriticalPointTh68}).}. A {\em $n$-manifold}
will be a finite dimensional Banach manifold with dimension $n\in
\N$.  When the infinite dimension is allowed, we will remark
explicitly that $M$ is Banach (say, modelled on some Banach space
$\EE$ with norm $\parallel \cdot
\parallel$) or, when applicable, that it is Hilbert (modelled on some real Hilbert space $H$ with
inner product $\langle \cdot,\cdot \rangle$). When indefinite
metrics are considered, as in Section \ref{s4},  $M$ will be
typically a $n$-manifold.

\runinhead{Finsler Banach manifolds.} $F$  will denote a
(reversible) Finsler metric on the Banach manifold $M$, and
$(M,F)$ will be called a Finsler Banach manifold. This notion is
taken in the sense of Palais \cite{PalaisTop_LustSch66}, that is,
$F$ yields a norm at each tangent space:
\begin{equation}\label{eFp}
F_p: T_pM \rightarrow \R
\end{equation}
which admits a $C^k$ chart $(U,\varphi), p\in U, \varphi: U
\subset M \rightarrow \EE$ such that the induced norms
\begin{equation}\label{einducednorms}\parallel u
\parallel_q:= F_q(d(\varphi^{-1})_{\varphi(q)}(u)) \quad \quad \forall u\in
\EE,\end{equation} (here $d$ denotes  the  differential or tangent
map) satisfy: (a) they are equivalent to the natural norm
$\parallel \cdot
\parallel$ of $\EE$ (i.e., $\epsilon_q \parallel \cdot
\parallel_q \leq \, \parallel \cdot
\parallel \, \leq \epsilon_q^{-1} \parallel \cdot
\parallel_q$ for some $0<\epsilon_q<1$ and all $q\in U$) , and (b) they vary continuously at
$p$ (i.e., for each $0<\epsilon <1$ there exists a neighborhood
$U_\epsilon \subset U$ of $p$ such that $\epsilon \parallel \cdot
\parallel_q \, \leq  \parallel \cdot
\parallel_p \, \leq \epsilon^{-1} \parallel \cdot
\parallel_q$ for all $q\in U_\epsilon$).

\begin{figure}
 \centering
\includegraphics[scale=0.4]{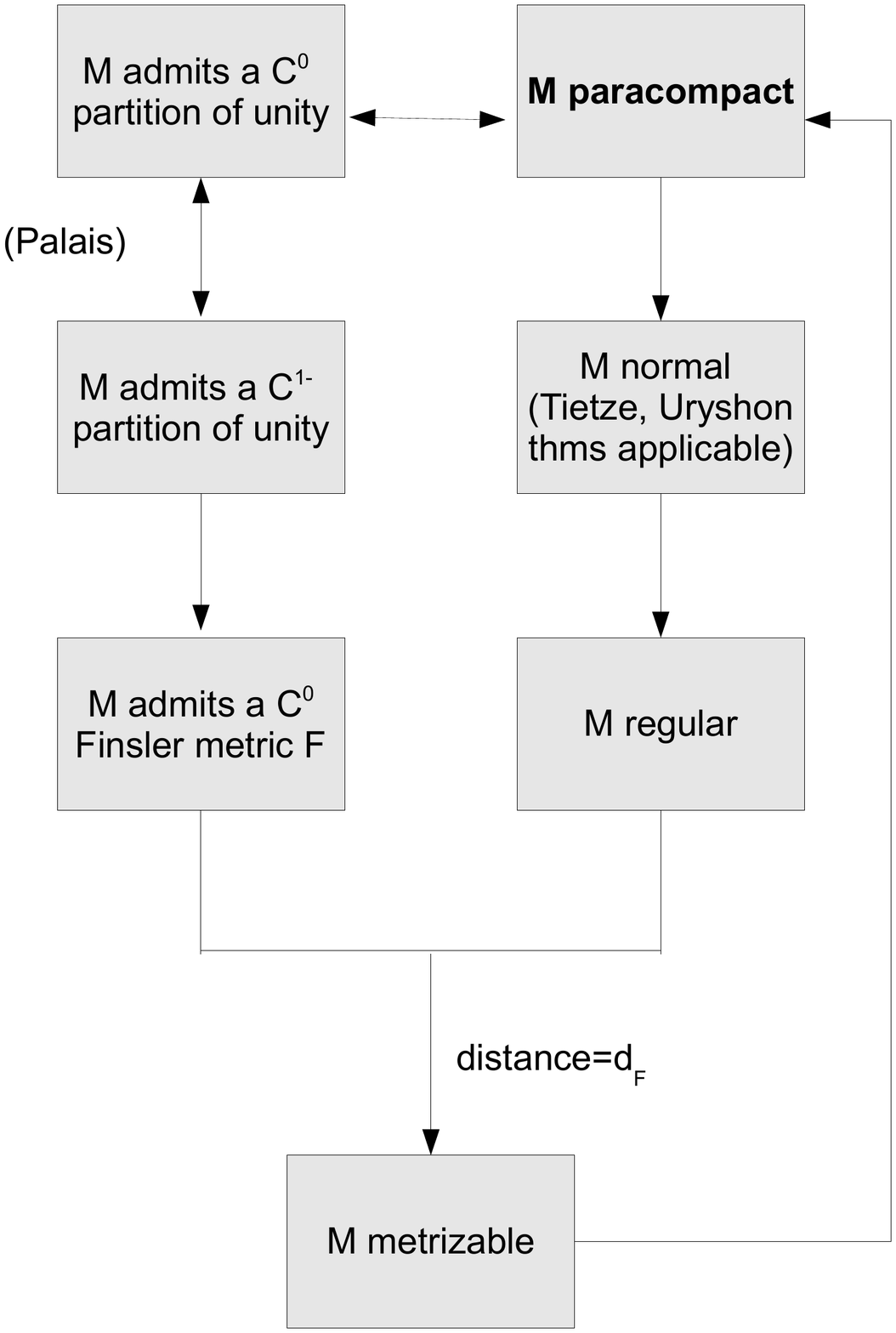}
\caption{Topological properties related to the paracompactness of
a (connected, Hausdorff) Banach manifold.}
  \label{fig1}
\end{figure}

As norms cannot be differentiable at\footnote{By the same reason
that neither is the absolute value function on $\R$. Moreover, at least in the
finite-dimensional case, the square of a norm is smooth at 0 if
and only if the norm comes from a scalar product \cite[Prop. 4.1]{Wa}.}
0,  the $C^{k'}$ differentiability of the norm $\parallel \cdot
\parallel$ means always away from $0$. The Finsler metric is called $C^{k'}$ (for $0\leq k'\leq k-1$)
if  $F_p$ is $C^{k'}$ and varies smoothly with $p$ in a $C^{k'}$
way (i.e., for any  chart $(U,\phi)$ as above the map $U\times
(\EE\setminus\{0\}) \rightarrow \R, (q,u)\mapsto
\parallel u\parallel_q$ is $C^{k'}$).

\runinhead{Existence of Finsler metrics.} The question of the
existence of a $C^0$ Finsler metric depends only on topological
grounds, but the existence of a $C^{k'}$ one with $k'>0$ is much
subtler. Namely, on the one hand the hypothesis of paracompactness
on $M$ becomes equivalent to the existence of $C^0$-partitions of
the unity subordinated to any open covering. By a result of Palais
\cite[Th. 1.6]{PalaisTop_LustSch66}, \cite[Sect.
3]{PalaisProcSym_CriticalPointTh68}, it is also equivalent to the
existence of locally Lipschitz partitions of the unity, and this
allows  ensuring the existence of $C^0$ Finsler metrics in any
Banach manifold \cite[Th. 2.11]{PalaisTop_LustSch66}.
On the other hand, when the model Banach space $\EE$ admits $C^k$
partitions of the unity subordinate to any open covering (which
happens, in particular, when $\EE$ is separable and admits a $C^k$
norm away from 0, see \cite{BF}, \cite[Prop. 5.5.18,
5.5.19]{AMR}), then the Banach manifold $M$ also admits $C^k$
partitions of the unity \cite[Th. 5.5.12]{AMR} and, in this case,
$M$ admits $C^{k-1}$-Finsler metrics too (Figure \ref{fig2}).
\begin{remark}\label{rproducts} It is worth pointing out that, even though the
differentiability of $F$ may be useful for some issues (see
Section \ref{s3.2} below),
it will not be especially relevant for the estimates which involve
length or distances in the first order problems  to be studied in
Section \ref{s2.1}.  This fact is used implicitly in
time-dependent problems. In fact, this case is commonly handled by
transforming it into a non time-dependent one, defined on the product
manifold $M\times \R$ which is endowed with a natural  direct sum of
Finsler metrics (namely, the addition of the Finsler metrics of the
factors), see Remark~\ref{rtimedependentX}. Nevertheless, this
direct sum is non-differentiable away from 0 even if
differentiability is assumed for the metric on each factor (notice
that it is not guaranteed the differentiability  on a vector
tangent to the product whenever one of its two components is equal to
zero).
\end{remark}

\begin{figure}
  \centering
\includegraphics[scale=0.4]{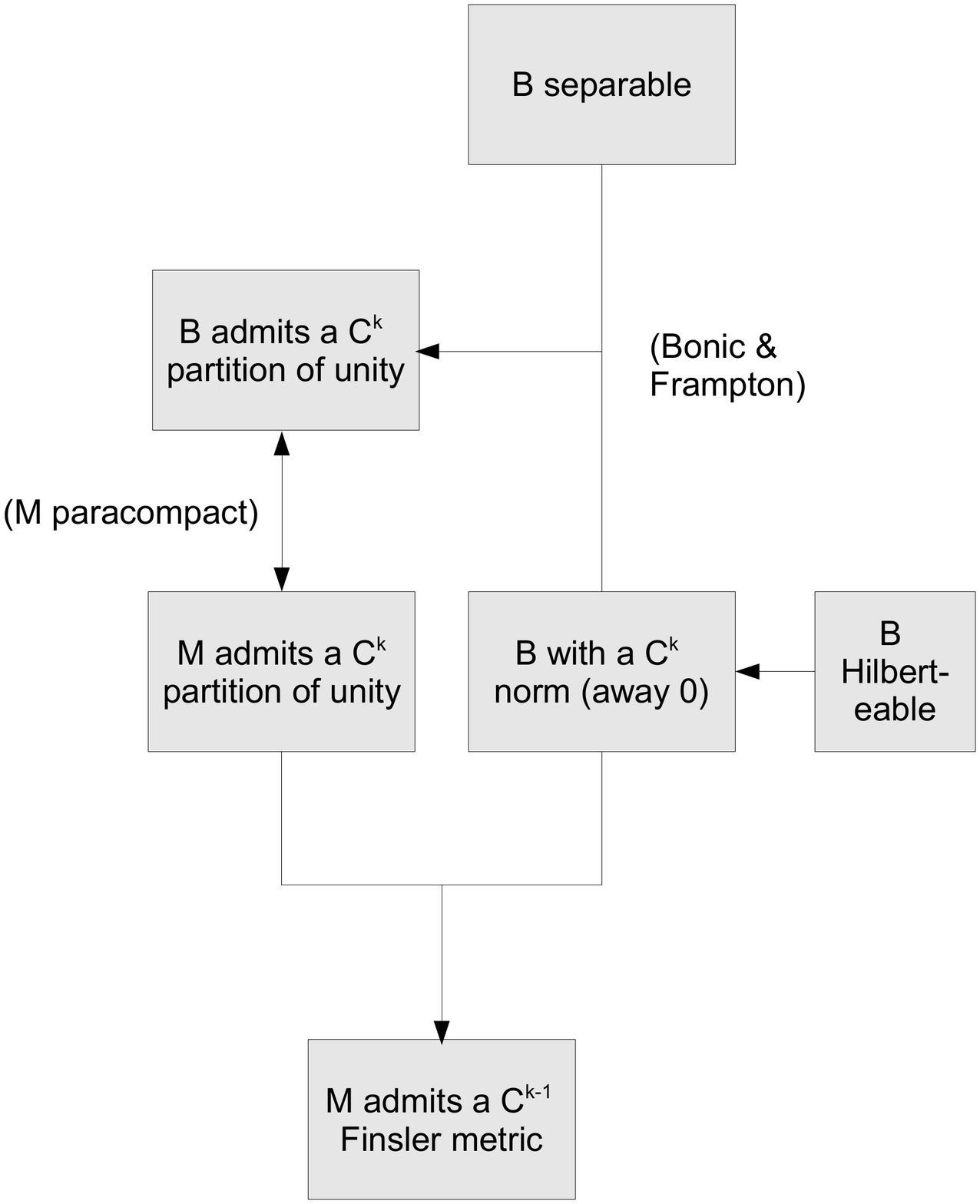}
\caption{Existence of smooth Finsler metrics on a manifold $M$
modelled on the Banach space $B$.}
  \label{fig2}
\end{figure}

\runinhead{Associated distance.} Remember that our definition of a
Finsler metric $F$ includes its {\em reversibility} (i.e.,
$F(v)=F(-v)$ for all tangent vector $v\in TM$). So, $F$
defines a natural distance
 by taking the infimum of the lengths of the curves connecting
each pair of points. This distance will be denoted $d_F$ or,
simply, $d$ if there is no possibility of confusion. One can
prove that the topology generated by $d$ agrees with the manifold
topology by using the regularity of the manifold.
\cite[p. 202]{PalaisProcSym_CriticalPointTh68}
and, so, that all Finsler Banach manifolds are
metrizable\footnote{Consistently, paracompactness can be deduced
from the hypothesis of metrizability (or even just from
pseudo-metrizability, see \cite[Lemma 5.515]{AMR}).}.

We will speak about the completeness of $(M,F)$ in the sense of metric completeness, i.e., the
convergence of  Cauchy sequences for $(M,d)$. One can also
consider geodesics for $(M,F)$ (for example, in the sense of locally length-minimizing
curves of constant speed, with other characterizations under
further smoothability, see Section \ref{s3.2}) and we will say
that $(M,F)$ is {\em geodesically complete} when  its inextensible
geodesics are defined on all $\R$. In the infinite-dimensional case, the completeness of $(M,F)$
implies geodesic  completeness  but, as stressed by Atkin
\cite{At}, neither geodesic completeness implies metric completeness nor other consequences of Hopf-Rinow theorem hold.

In order to make estimates with the distances, we fix a {\em base}
point $p_0\in M$ and denote
\begin{equation}\label{enotationdistancetop0}
|p| = d(p,p_0) \quad \quad \forall p\in M. \end{equation} (This
notation will be used when  the properties under study are
independent of the chosen point $p_0$).

\runinhead{Pseudo-Riemannian metrics on 
Banach
manifolds.}
 When the model space $\EE$ of the
Banach manifold $M$ is reflexive, it is natural to define a
$C^{k'} (k'\leq k-1)$ pseudo-Riemannian metric $g$  as a $C^{k'}$
choice of a continuous symmetric bilinear form $g_p$ at each
tangent space $T_pM$ such that the associated ``flat''
 map (to
lower indexes in finite dimension) into the dual space given by
\begin{equation}\label{ebemol}\flat_p: T_pM\rightarrow T_pM^*, \quad v_p\mapsto g_p(v_p,\cdot
)\end{equation} is a homeomorphism
(if this condition on
$\flat_b$ were not imposed, one would speak of a {\em weak}
pseudo-Riemannian metric, and the reflexivity of $\EE$ would not be
required). The set of all such bilinear forms $g_p$ can be
identified via a chart around $p$ with an open subset of the set
BL$_{sym}(\EE)$ of all the continuous symmetric bilinear forms on
$\EE$. As BL$_{sym}(\EE)$ is naturally a Banach space too, the
pseudo-Riemannian metric $g$ can be regarded as a section of a
fiber bundle on $M$ with fiber BL$_{sym}(\EE)$ (see \cite[Ch
VII.1]{Lang}).

\runinhead{Riemannian metrics on Hilbert manifolds.} When the
pseudo-Riemannian metric $g$ is positive definite then we say that
it is Riemannian. As we are assuming that $\flat_p$ is a
homeomorphism,  the model space $\EE$ is then Hilberteable.
So, it will be denoted $H$, and we will consider only
Riemannian metrics on Hilbert manifolds.
Notice that, for any Riemannian
metric $g$, one has an associated Finsler metric given by
$F(v)=\sqrt{g(v,v)}$ for all $v\in TM$. So, the bounds required in
the definition of continuity for $F$ in the Finslerian case (see
(a) and (b) below formula (\ref{einducednorms})), hold here in
terms of the norm associated to the inner product $\langle \cdot,
\cdot \rangle$ of $H$. Moreover, this norm is always $C^\infty$
away from 0. Thus, any $C^k$ Hilbert manifold modelled on a separable
space $H$ admits $C^k$ partitions of the unity and, then, a
$C^{k-1}$ Riemannian metric. Riemannian metrics on Hilbert
manifolds, as well as their geodesics, are extensively studied in
the literature, see for example \cite{Lang} or, for the separable
case, \cite{Kl}. A
type of Hopf-Rinow theorem for separable Riemann
Hilbert manifolds can be found   in \cite[Th. 2.1.3]{Kl}
(including the ``Notes'' therein; recall also   \cite{At}); some related remarkable properties
can be seen in \cite{Ek}.

\runinhead{Concluding remarks and conventions.} For the
convenience of the reader, a summary on the topological and smooth-related
results commented above is provided in Figures~\ref{fig1}
and \ref{fig2}. Basic detailed background can be found in
\cite{PalaisTop_LustSch66, PalaisProcSym_CriticalPointTh68} and
\cite{AMR}. In what follows, all the objects will be {\em smooth}
i.e. as differentiable as possible according to the  discussion
above.
In the case of first
order problems (Section \ref{s2.1}), this will mean at least $C^2$ for any Banach manifold $M$ and
$C^1$ for any
 vector field $X$ on $M$. As emphasized in Remarks
\ref{rproducts} and \ref{rtimedependentX},  Finsler metrics are
required only $C^0$ at this stage. Further requirements of
smoothability will be needed for the  second order case (Section
\ref{s2.3}). In the (indefinite) finite-dimensional case (Section
\ref{s4}), the issues on smoothability are not especially relevant
and, so, the reader may either track them or just assume
$C^\infty$ smoothability.

\section{Completeness of trajectories in a positive-definite infinite-dimensional setting}
\label{s2}

This section is divided into two subsections. The first one is
devoted to the  problem of the completeness of a
vector field. We start by reviewing some results. These have essentially been known from the
seventies \cite{Go, Eb, WM} and explained in   \cite{AM,
AMR}. They are extended here to the ($C^0$) Finsler setting when possible
(Propositions \ref{pextend}, \ref{pcompleteboundedsubintervals}).
Then, the notion of {\em primarily complete} function is introduced
(Definition~\ref{dprimarycomplete}).
Primary
bounds for a vector field
allows us to give an optimal result on completeness in the Finsler Banach case,
Theorem \ref{tsublinear}. The time-dependent case is specially
discussed in Remark \ref{rtimedependentX} and the last part of the
subsection.

In the second subsection, our Theorem \ref{t1} (plus Remark
\ref{remarkth}) summarizes and extends the results on second order
differential equations in \cite{Go, Eb, WM, CRS}. The proof is
carried out in three conceptually independent steps. The first one is just
a standard reduction to the first order case. The second one deals with
technical bounds. This is carried out here just by using
systematically the simple lemma \ref{lsubsol}. In the third step, the subtleties
of the infinite dimensional case (first studied by Ebin \cite{Eb}), are stressed.

Further discussions are also provided in this second subsection. Firstly, the relation
between the previous notion of {\em primarily complete} function
and Weinstein-Marsden's  {\em positive completeness}, is
analyzed. Secondly, we consider specifically the time-dependent case.
Even though natural bounds are obtained
for the growth of the potential in
this case, we also explore some alternatives. Finally, we 
discuss the  difficulties of the generalization when the Riemannian
metrics are replaced by Finslerian ones, and we provide a simple
example  for the (standard) finite-dimensional Finsler
case.

\subsection{Complete vector fields on Finsler Banach manifolds}
\label{s2.1}

\subsubsection{Elementary criteria}

 The properties of the (local) flow $\phi$ of a vector field $X$ and, in
particular, the existence of a flow box around each point, can be
found, for example, in \cite[p. 192ff]{AMR}, \cite[p. 84ff]{Lang}
or \cite[Sect. 1.10]{Moore}. We start with a well-known result
(see for example \cite[Prop. 4.1.19]{AMR}).

\begin{proposition}\label{pextend}
Let $X$ be a  vector field on  a Banach manifold $M$, and let $c:
[0,b) \to M$ (resp. $[-b,0) \to M$) be an integral curve of $X$
with $0<b<+\infty$. Then, $c$ can be extended beyond $b$ as an
integral curve of $X$ if and only if there exists a sequence $t_n
\to b^-$ such that the sequence $\{c(t_n)\}_n$ (resp.
$\{c(-t_n)\}_n$) is convergent in $M$.
\end{proposition}
\begin{proof}
The necessity of the condition is obvious. For its sufficiency,
let $p\in M$ be the limit of the sequence. The existence of a flow
box of $X$ at $p$  ensures the existence of a neighborhood $U$ of
$p$ and some $\epsilon>0$ such that the integral curves of $X$ at
any $p'\in U$ are defined on $(-\epsilon, \epsilon)$. So, taking
$n$ large so that $b-t_n<\epsilon$ the integral curve through
$c(t_n)$ will be defined on $[0,\epsilon)$ and $c$ will be
extensible through $b$.
\end{proof}
Accordingly, we will say that an integral curve $c$ of $X$ defined
on some interval $I$ of $\R$ is {\em complete} if it can be
extended as an integral curve of $X$ to all $\R$, and $X$ will be
complete if so are  its integral curves.

\begin{remark} (i) This result follows in the infinite-dimensional case
as
well as in the finite-dimensional one. However, the application in
the latter case is easier, as $M$ is then locally compact.
For example, Proposition \ref{pextend} yields directly that, if
the support of $X$ is compact (in particular, if $M$ is compact
and, thus, finite-dimensional) then $X$ is complete.

(ii) Analogously, one can prove that if a  Banach manifold $(M,F)$
admits a $C^1$-proper map $f: M\rightarrow \R$ (i.e.
$f^{-1}([a,b])$ is compact for any compact  $[a,b]\subset \R$),
then a vector field $X$ is complete whenever
\begin{equation}\label{eproper}
|X_p(f)| \leq C_1 |f(p)| +C_2 \end{equation} for some $C_1,C_2>0$
and all $p\in M$. In fact,  (\ref{eproper}) implies a bound for
the derivative of $\log(C_1|f\circ c|+C_2)$. If the domain of
the integral curve $c$ were bounded,   a bound for $f$ on $c$ would be obtained too. As $f$ is proper, the result would follow then from Proposition \ref{pextend} (see \cite[2.1.20]{AM} or
\cite[4.1.21]{AMR} for more details). Even though proper maps are
well behaved in Banach manifolds (for example, they are closed
maps \cite{PalaisProcAMS_WhenProperMapsAreClosed}) results as the
previous one are used typically in the finite-dimensional case
(putting, for example, $f=C_1|x|^2+C_2$ on a complete Riemannian
$n$-manifold).
\end{remark}
The following criterion on completeness for Finsler Banach
manifolds holds as in the case of Riemann Hilbert ones or Banach
spaces (compare with \cite[Prop. 2.1.2]{AM} or \cite[Prop.
4.1.22]{AMR}).

\begin{proposition}\label{pcompleteboundedsubintervals} Let $(M,F)$ be a complete Finsler Banach
 manifold and $X$  a  vector field on $M$.
 If $c: I\subset
\R\rightarrow M$ is an integral curve of $X$ and $F(\dot c)$ is
bounded on bounded subintervals of $I$, then $c$ is complete.

\end{proposition}
\begin{proof}
Assume with no loss of generality that $I=[0,b), b<\infty$,  let
$A$ be the assumed bound and choose $\{t_n\}\nearrow b$. The
associated distance $d$ satisfies then:
$$
d(c(t_n),c(t_m)) \leq \int_{t_n}^{t_m}F(\dot c(t))dt \leq
A|t_n-t_m|.
$$
So, $\{c(t_n)\}_n$ is a Cauchy sequence, which becomes convergent
to some limit $p$ by the completeness of $(M,F)$. Then,
Proposition \ref{pextend} can be applied to $\{c(t_n)\}_n$.

\end{proof}

\begin{remark}\label{rtimedependentX} ({\em The time-dependent case.})
The results in the previous two propositions can be extended to
the case when $X$ is  time-dependent, and defined for all the
values of the time.

More precisely, consider
 the product manifold $M\times \R$, let $\Pi_\R:M\times \R\rightarrow \R$, $\Pi_M:M\times \R\rightarrow
 M$ be the natural projections, and denote by $t$ the natural coordinate
 on $\R$. We say that $X$ is a {\em time-dependent vector field on} $M$
 if it is a smooth section of the pull-back bundle  $\Pi_M^\star(TM)$,
whose
 base is $M\times \R$ and each fiber comes from a tangent space to $M$.
Such a vector field yields naturally a (time-independent) vector
field $\hat X$ on $M\times \R$ which satisfies both, $d\Pi_M\hat
X_{(p_0,t_0)}$ is naturally identifiable to the natural
projection of $X_{(p_0,t_0)}$ on $T_{p_0}M$ and $d\Pi_\R\hat
X_{(p_0,t_0)} =
\partial_t|_{t_0}$, for all $(p_0,t_0)\in
M\times~\R$.

To speak about the integral curves of $X$ makes a natural sense (see
for example \cite[Ch. IV]{Lang}) and becomes equivalent to
consider the integral curves of $\hat X$; in fact, $c$ will be an
integral curve of $X$ if and only if $\hat c: t\mapsto (c(t),t)$
is an integral curve of $\hat X$. So, Proposition \ref{pextend} is
extended directly to a time-dependent $X$.

To extend Proposition \ref{pcompleteboundedsubintervals}, recall that,
if $(M,F)$ is a Finsler Banach manifold, then $M\times \R$ admits
a natural
 $C^0$ Finsler metric $\hat F$ obtained as the direct sum of $F$ and the usual one on
 $\R$ (see Remark \ref{rproducts}). Clearly, $\hat F$ will be complete if and only if so is
 $F$. Moreover, the $F$-length of the integral curve $c$ of $X$ is bounded
 on finite intervals if and only so is the $\hat F$-length of the integral curve $\hat c$ of $\hat
 X$, as required.
 \end{remark}

\subsubsection{Applications} Next, we will apply previous results to
simple but general situations. But, previously, we consider the following technical
elementary result for future referencing (see for example \cite[Lemma 1.1]{Te}).
\begin{lemma}\label{lsubsol}
Consider the equation
\begin{equation}\label{sub1}
\dot u\ =\ f(t,u) \quad  \hbox{on} \quad  \hbox{$[t_0, T)$,}
\end{equation}
where $f \in C^0(\R^2,\R)$ is locally Lipschitz in its second variable, 
and let $w=w(t)$ be a subsolution of the differential equation
i.e.,
\begin{equation}
\label{elema1}
\dot w(t) < f(t,w(t)) \qquad   \forall t\in [t_0, T).
\end{equation} Then, for every solution
$u=u(t)$ of (\ref{sub1}) such that $w(t_0) \le u(t_0)$ we have
\begin{equation}
\label{elema2}
w(t) < u(t) \quad \hbox{for all} \quad  t \in (t_0, T).
\end{equation}
The same conclusion \eqref{elema2} holds if $w$ is only locally
Lipschitzian and  the inequality \eqref{elema1} occurs when $\dot w(t)$
is replaced  by some local Lipschitz bound around each $t$.
\end{lemma}
The proof follows just recalling that $\Delta:=w-u<0$ close to
$t_0$ by the assumptions and, if there were a first point such
that $\Delta(t_1)=0$, then $\dot\Delta(t_1)<0$ (or an analogous
inequality involving a local Lipschitz bound) holds, a contradiction.


\runinhead{Estimates of the growth for completeness.} Let us
introduce some auxiliary definitions.

\begin{definition}\label{dprimarycomplete}
A (locally Lipschitz)  function $\alpha: [0,\infty)
 \rightarrow R$ is {\em primarily complete} if it is
 positive, non-decreasing  and satisfies:
\begin{equation}\label{esublinealpha} \int_0^{\infty}\frac{dx}{\alpha(x)}=\infty
\end{equation}

A vector field $X$ on a Finsler Banach manifold $(M,F)$ is {\em
primarily bounded} if there exists a primarily complete function
$\alpha$, which be called a {\em bounding function}, such that
\begin{equation}\label{esubalpha}
 F(X_p)< \alpha(|p|))\quad \quad\quad\quad \forall p\in M.
\end{equation}
In particular, $X$ {\em grows at most linearly} if it is primarily
bounded by an affine bounding function, i.e.:
\begin{equation}\label{esublineal} F(X_p)<C_0+C_1|p| \quad\quad\quad \forall p\in
M,
\end{equation} for some constants $C_0,C_1>0$.
\end{definition}

\begin{remark}\label{rpositivelybounded}
The best polynomial candidate for the bounding function $\alpha$ has
degree one as, clearly, no polynomial of higher degree can be a
primarily complete function. Nevertheless, a slightly faster
growth is allowed for non-polynomial functions. For example,
$\alpha$ will be primarily complete if it grows as $x \cdot
\log x \cdot \log (\log x)$ for large $x$ (see also the discussion
in the last part of Section \ref{s3.2.1}).
\end{remark}

Now, we can give a general bound for the completeness of vector
fields.
\begin{theorem}\label{tsublinear} Any primarily bounded vector field on   a complete Finsler Banach
 manifold  $(M,F)$ is complete.

\end{theorem}
\begin{proof}
Let  $c: I\rightarrow M$ be an integral curve of $X$. With no loss
of generality, assume $I=[0,b)$,  $0\leq t_0<t_1<b$,  and  choose $p_0=c(0)$ in the
notation introduced in (\ref{enotationdistancetop0}). Then:
\begin{equation}\label{ealpha}
||c(t_1)|-|c(t_0)||\leq \int_{t_0}^{t_1}F(\dot c(s))ds < \int_{t_0}^{t_1} \alpha(|c(s)|)ds,
\end{equation} where $\alpha$ is the bounding function. Thus, putting $w(t)= |c(t)|$ we can assume: $$\dot w(t) <
\alpha(w(t)),
$$
(or the analogous inequality for  local Lipschitz constants). The unique inextensible solution $w_0$ of the equality
$$\dot w_0(t)=\alpha(w_0(t)) \quad \quad w_0(0)=w(0)(=0)$$
is defined for all $t\in [0,\infty)$, as its inverse is determined
as $w\mapsto t(w)=\int_0^w d\bar w/\alpha(\bar w)$ and
(\ref{esublinealpha}) holds. So, from lemma \ref{lsubsol} one has
$$w(t)<w_0(t)<w_0(b) \quad \quad \forall t\in (0,b).$$ As $\alpha$
is non-decreasing, equation (\ref{esubalpha}) yields the  bound
$F(\dot c) \leq \alpha(w_0(b))$ so that Proposition
\ref{pcompleteboundedsubintervals} is applicable.
\end{proof}
\begin{remark}
By considering on $\R$ a vector field type
$X_{x_0}=\alpha(x_0)\partial_x$ one can check the optimality of
Theorem \ref{tsublinear} and, in particular,  the optimality (in
the sense discussed in Remark \ref{rpositivelybounded}) of the at
most linear growth of $X$ to ensure completeness. Of course, a
vector field with a {\em superlinear} growth such as
$X=y^2\partial_x$ may be complete. In fact, in order to ensure
completeness,  only
 the growth of $X$ along the direction of its integral
curves becomes relevant. This underlies in the fact that the sum
of two complete vector fields $X,Y$ may be incomplete (put $
Y=x^2\partial_y$ and $X$ as before) and may suggest more refined
hypotheses for completeness in Hilbert spaces (compare with
\cite[Exercise 2.2H]{AMR}).
\end{remark}

\runinhead{Time-dependent case.} As in the case of the criterions
on completeness, Theorem~\ref{tsublinear} can be extended  to the
case of a time-dependent vector field $X$. In fact, the proof
works in a completely analogous way (with the observations in
Remark \ref{rtimedependentX}), if the inequality in
(\ref{esublinealpha}) is  regarded as $F(X_{(p,t)})< \alpha(|p|)$
for all $(p,t)\in M\times \R.$ Nevertheless, one can be a bit more
accurate.

\begin{definition} A time-dependent vector field $X$ on a Finsler Banach
manifold is {\em primarily bounded along finite times} if there
exists a primarily complete function $\alpha$ and a continuous
function $C(t)>0$ such that
$$F(X_{(p,t)})< C(t) \alpha(|p|) \quad \quad \quad \quad \forall (p,t)\in M\times \R .$$
In particular, $X$ {\em grows at most linearly along finite times}
when $\alpha$ can be chosen affine or, equivalently, when
\begin{equation}\label{esublinealatfinitetimes} F(X_{(p,t)})<C_0(t)+C_1(t)|p| \quad\quad\quad \forall (p,t)\in
M\times \R
\end{equation} for some functions $C_0(t), C_1(t) >0$
\end{definition}

\begin{corollary}
Let $X$ be a time-dependent vector field on a complete Finsler
Banach $(M,F)$. If  $X$ is primarily bounded along finite times
then it is complete.
\end{corollary}
\begin{proof}
Reasoning  with an integral curve $c$ defined on $[0,b)$ as in the
proof of Proposition \ref{pcompleteboundedsubintervals}, notice
that the inequality (\ref{esublinealatfinitetimes}) for all the
pairs $(p,t)\in M\times [0,b]$ also yields a time independent
inequality as (\ref{esublineal}) with
 $C_i=$Max$_{t\in
[0.b]}\{C_i(t)\}, i=0,1$. Then, reason as in Remark
\ref{rtimedependentX} taking into account that $\hat X$ is
primarily bounded (on $M\times [0,b]$) if and only if so does
 $X$.
\end{proof}

\subsection{Completeness for 2nd order trajectories}\label{s2.3}

\subsubsection{General result on  Riemann Hilbert manifolds}
\label{s3.2.1}

 The next result, stated on a Riemann Hilbert manifold $(M,g)$, will summarize those
in \cite{Go,Eb,AM,CRS}. To state it, recall that the notion of
time-dependent vector field on $M$ in Remark \ref{rtimedependentX}
can be directly translated to (continuous, linear)  endomorphism
fields, which will be regarded here as sections of a fiber bundle on $M\times \R$
with fiber at each $(p,t)\in M\times \R$ equal to the vector space
of bounded linear operators $T_{(p,t)}(M\times \R)\rightarrow
T_{(p,t)}(M\times \R)$ which vanish on $(0,\partial_t)_{(p,t)}$.
Given such a field $E$, we will decompose it as $E=S+H$ where $S$
denotes its self-adjoint part $(S =(E+E^\dagger) /2)$, and $H$ the
skew-adjoint one.
 A time-dependent or non-autonomous
potential  means just a smooth map $V:M\times \R\rightarrow \R$,
then, the notation $\partial V/\partial t :M\times \R\rightarrow
\R$ makes a natural sense, and $\nabla^MV$  denotes the time
dependent vector field on $M$ obtained by taking the gradient of
$V$ at each slice $t=$constant with respect to $g$, i.e.,
$dV(X(p,t),0)=g_p(\nabla^MV(p,t),X(p,t))$ for $(X(p,t),0)\in
T_{(p,t)}(M\times \R)$.
The pointwise norm induced by $g$ in any space of tensor fields
will be denoted $\parallel \cdot \parallel$.

\begin{theorem}\label{t1}
Let $(M,g)$ be a complete Riemann Hilbert manifold, and consider a
endomorphism field $E=S+H$, a vector field $R$ and
a potential $V$ on $M$, all of them time--dependent and smooth.
Assume that:

\begin{description}[(i)]
\item[(i)] $S$ is uniformly bounded along finite times, i.e.,
$\parallel S_{(p,t)} \parallel \leq C_0(t)$ for all $(p,t)\in
M\times \R$,

\item[(ii)] $R$ grows at most linearly along finite times, i.e.,
$\parallel R_{(p,t)} \parallel \leq C_0(t) + C_1(t) |p|$ for all
$(p,t)\in M\times \R$, and

\item[(iii)] both, $-V$ and $|\partial V/\partial t|$ grow at most
quadratically along finite times, i.e., they are bounded by
$C_0(t)+ C_2(t)|p|^2$,

\end{description}  where $C_i(t), i=0,1,2$, denote positive
functions.  Then,  the inextensible solutions of
\begin{equation}\label{e}
\frac{D\dot\gamma}{dt}(t)\ = E_{(\gamma(t),t)}\
\dot\gamma(t) + R_{(\gamma(t),t)}  - \nabla^M V (\gamma(t),t), 
\end{equation}
are complete.
\end{theorem}

\begin{proof} In order to clarify the ideas, the proof is divided into three steps.

{\em Step 1: Reduce the problem to the completeness of a vector
field on the tangent bundle.} The  second order equation (\ref{e})
allows to define a vector field $G$ on the manifold $T(M\times
\R)$ such that each solution $\gamma$ of (\ref{e}) generates an
integral curve $t\mapsto (\gamma'(t),1)$ of $G$. This is standard
(see for example \cite[Ch. 3]{AM} or, for explicit details on the
time-dependent case, \cite[Section 3.1]{CRS}) and, so, the problem
will be reduced to apply the criterions in Propositions
\ref{pextend} and \ref{pcompleteboundedsubintervals} to $G$.

{\em Step 2: Find a bound for the velocity of any solution $\gamma$ of
(\ref{e}), by using the hypotheses (i) to (iii).} With no loss of
generality, let $\gamma: [0,b) \rightarrow M, b<\infty$, be a
solution of (\ref{e}) whose extendability to $b$ is to be
determined,
 let $u(t)=g(\dot \gamma(t),\dot \gamma(t))$ the
function to be bounded, and choose the base point
 $p_0=\gamma(0)$ for (\ref{enotationdistancetop0}).
Taking in (\ref{e}) the product  by $\dot \gamma$:
$$
\frac{1}{2}\dot u(t) =\ g(S_{(\gamma(t),t)}
\dot\gamma(t),\dot\gamma(t)) + g(R_{(\gamma(t),t)},\dot\gamma(t))
- \left(\frac{d}{dt}V(\gamma((t),t)-\frac{\partial V}{\partial
t}(\gamma((t),t)\right)
 $$
so that taking pointwise norms and  simplifying the notation:
\begin{equation}\label{e1}
\begin{array}{ll}
\frac{d}{dt}(\frac{1}{2}u+ V) &\leq  \parallel S \parallel u +
\parallel R
\parallel \sqrt{u} + \partial V/\partial t \\ &
\leq (\parallel S \parallel + 1/2) u + \parallel R
\parallel^2/2 + \partial V/\partial t
\end{array}\end{equation}
Using the bounds (i), (ii), (iii) and taking into account that, as
the $t$ coordinate is confined in  the compact interval $[0,b]$,
the $t$-dependence of these bounds can be dropped:
\begin{equation}\label{e2} \frac{d}{dt}(u+2V) \leq A_0 + A_1 u
+ A_2 |\gamma|^2 \end{equation} for some constants $A_0, A_1,
A_2>0$. Consider the function $l(t)=\int_0^t\sqrt u$, $t\in[0,b)$
which provides the length of $\gamma$. Clearly:
$$
|\gamma(t)|^2 \leq l(t)^2 \quad \quad \hbox{and} \quad \quad
\int_0^t l(\bar t)^2 d\bar t \leq b \cdot l(t)^2 \quad \quad
\forall t\in [0,b), $$
the latter as $l$ is
nondecreasing. Using these inequalities and integrating in
(\ref{e2}):
$$
u(t)-A_1\int_0^t u \leq A_0' -2 V(\gamma(t),t)+ A_2 b l(t)^2 <
C_0+ C_1 l(t)^2,
$$
where $A_0',C_0, C_1$ are constants ($C_0$ and $C_1$ positive),
obtained by taking into account the hypothesis (iii). So, putting
$v(t)=\int_0^t u$ and relabelling $A_1$,
\begin{equation}\label{e3}
\dot v < C_0 + C_1 \cdot l^2 + C_2 \cdot v \quad \quad \quad
\hbox{for some constants} \; \; \; C_0, C_1, C_2 >0.
\end{equation}
Now, $v$ can be regarded as a subsolution of a differential
equation, and lemma \ref{lsubsol} will be applicable to the
solution $v_0$ of this equation with $v_0(0)=v(0)=0$ i.e. $v(t) <
v_0(t)$ and, taking into account (\ref{e3}):
$$
 \dot v < C_0 + C_1 \cdot l^2 + C_2 \cdot v_0 = \dot v_0
$$
on $(0,b)$. As $u = \dot v$,  to bound $\dot v_0$ would suffice.

Notice that $v_0$ can be written  explicitly as:
$$
v_0(t)= e^{C_2 t}\int_0^t e^{-C_2 \bar t}(C_0+C_1 l(\bar
t)^2)d\bar t
$$
so that, using that $l$ is nondecreasing,
\begin{equation}\label{enu}
\dot v_0 \leq C_0 + C_1 l^2 + C_2 b e^{C_2  b}(C_0 + C_1 l^2)= A+
B l^2 \quad \quad \quad \hbox{on} \; [0,b) \end{equation}
 for some constants $A, B >0$. But
recall that $\dot l = \sqrt{u} <\sqrt{\dot v_0}$, that is, $l$ can
be also regarded as a subsolution of a differential equation:
\begin{equation}\label{ele}
\dot l < \sqrt{A+B\cdot l^2}. \end{equation} So, $l$ is bounded by
the corresponding solution ($l(t)< \sqrt{A/B} \cdot \sinh(\sqrt{B}
\cdot t $ on $(0,b)$) and, thus, $u$ (regarded either as $\dot
l^2$ in (\ref{ele}) or as  $\dot v_0$ in (\ref{enu})) is bounded,
as required.

{\em Step 3: As $g$ is complete, $\dot \gamma$ must lie in a
compact subset}. The aim is to prove the extendability of $\dot
\gamma$ as an integral curve of the vector field $G$ on
$T(M\times\R)$ defined in the first step. As a first consequence
of the boundedness of $u$, the completeness of $g$ imply that
$\gamma$ must be convergent in $M$. Then, it is convenient to
distinguish two type of reasonings:

(3a) In the case that $M$ is finite dimensional, the convergence
of $\gamma$ at $b$,  the boundedness of $u=g(\dot
\gamma,\dot\gamma )$ and the local compactness of $TM$, are enough
to ensure that $\dot\gamma$ lies in a compact subset of $TM$, so
that Proposition \ref{pextend} is applicable to $G$.

(3b) In the infinite-dimensional case, the lack of local
compactness requires a more elaborated argument. First, the
Riemannian metric $g$ on $M$ induces naturally a Riemannian metric
$\tilde g$ on $TM$, the {\em Sasaki metric} \cite{Sa}. As proven
by Ebin \cite{Eb}, $\tilde g$ is complete whenever so is $g$. The
vector field $G$ can be written as a sum $G=G_0+G_1+G_2$ where
$G_0$ is the geodesic spray and, thus, a horizontal vector field,
$G_1$ is a vertical vector field such that, at each $v_{(p,t)}$,
depends only of the value of $R+\nabla^M V$ at $(p,t)$ and $G_2$
is also a vertical vector which, at each $v_{(p,t)}$, can be
identified with $E(v_{(p,t)})$. The convergence of $\gamma$ yields
a bound for $\tilde g(G_1,G_1)$ on $\dot \gamma$, the boundedness
of $u$ implies a bound for $\tilde g(G_0,G_0)$ and, then, the
boundedness of the operator $E$ implies the boundedness of $\tilde
g(G_2,G_2)$. So, $G$ is bounded on $\dot \gamma$, and Proposition
\ref{pcompleteboundedsubintervals} is applicable.

\end{proof}

\begin{remark}\label{remarkth}
(1) The result can be also sharpened, if one is only interested in
the forward or backward  completeness of the trajectories
({\em positive} or {\em negative completeness}), i.e.
the possibility to extend the solutions  to an upper or lower
unbounded
 interval type $[a,\infty)$ or $(-\infty,a]$.
From the proof is clear that, in order to obtain the extensibility
of the trajectories to $+\infty$ (resp. $-\infty$), one requires
only the  upper (resp. lower) uniform bound of $g(v,S(v))/g(v,v)$,
for  $v\in TM\setminus\{0\}$,\footnote{This can be rephrased as a
bound of the spectrum of $S$, see \cite[Th. 3.10]{Lang}.} as well
as the upper (resp. lower) bound of $\partial V/\partial t$,
instead of the bounds for the norm and absolute value imposed in
the hypotheses (i) and (iii) .

(2) As a trivial consequence of Theorem \ref{t1},   if $M$ is
compact then all the inextensible trajectories are complete,
 for
any $E, R, V$.

\end{remark}

\runinhead{Primary and positively complete functions.} The optimal
growth allowed  either for $-V$ or for $|\partial V/\partial t|$
can be sharpened, by using bounds in the spirit of the {\em
primary} ones, introduced for Theorem \ref{tsublinear}, which are
clearly related to the notion of {\em positive completeness}
introduced by Weinstein and Marsden \cite{WM}.

Recall that a smooth function $V_0: [0,\infty) \rightarrow \R$ is called {\em
positively complete} if it is non--increasing and satisfies
$$\int_0^{+\infty} \frac {ds}{\sqrt{e - V_0(s)}}\ = \infty,$$
for some (and then all) constant $e
> V_0(0)$ (hence $e > V_0(s)$ for all $s \in [0,+\infty)$) \cite{WM, AM}.
 Extending Weinstein-Marsden notions, we say that a
smooth time-dependent function $V: M\times \R \rightarrow \R$ is
bounded by a positively complete function along finite times if
there exists functions $V_0, C: [0,\infty)\rightarrow \R$, $V_0$
positively complete and $C>0$ such that:
$$
V(p,t) \geq C(t) V_0(|p|)) \quad \quad \forall (p,t)\in M\times \R
.
$$
The relation between these notions and those used in the last
subsection comes from the fact that {\em a smooth function $V_0$ is
positively complete if and only if $\sqrt{e-V_0}$ is well-defined
and primarily complete for some} $e>V_0(0)$. Now, from the proof of
Theorem \ref{t1}, one can  easily check:

\begin{svgraybox} Hypotheses (ii) and (iii) in Theorem \ref{t1} can be replaced by
the following more general one: there exists a primarily complete
function $\alpha$ and a positive one $C$ such that  $R$ is
primarily bounded along finite times
by  $C\cdot \alpha$ 
and $-V(p,t), |\partial V/\partial t|(p,t) < C(t)^2\alpha(|p|)^2$
for all $(p,t)\in M\times \R$.
\end{svgraybox} In particular, the  quadratic bounds  in
(iii) can be improved by requiring only bounds\footnote{These
improvements can be also extended to other contexts, as the
completeness of certain Finler metrics in \cite{DPS}.} by, say,
$C_0(t)+ C_2(t)|x|^2\log^2(1+|x|)$ and the  linear bound in
(ii) by $\tilde C_0(t)+ C_1(t)|x|\log(1+|x|)$ (as well as by other functions
pointed out in \cite[p. 233]{AM} or \cite[Remark 5(2)]{CRS}).
These bounds might be optimized further, combining them also with
better bounds for $E$.

\runinhead{The time-dependence of the potential $V$.} For a
non-autonomous potential, the role of the bounds of $\partial
V/\partial t$ becomes quite subtler. Notice that one can regard
$\nabla^MV$ as a time-dependent vector. Thus:

\begin{svgraybox}
If we assume in Theorem \ref{t1} that $\nabla^MV$ grows at most
linearly along finite times, no bound for $\partial V/\partial t$
is necessary. Nevertheless, such a hypothesis is independent of
 the one  stated in {\em (iii)}. In fact, in the autonomous case,
if $\nabla^MV$ grows at most linearly then $-V$ grows at most
quadratically, but, clearly, the converse does not hold.
\end{svgraybox}
Other alternative bounds  for $\partial V/\partial t$ in Theorem
\ref{t1} can be explored. For example, assuming by simplicity
$R=0$ in {\em (ii)}, the result of completeness still holds if we
replace {\em (iii)} by the following two conditions:  $V$ is lower
bounded at finite times ($V(p,t)\geq -C_0(t)$) and:
\begin{equation}\label{eff}|\partial V/\partial t|\leq C_1(t)(V(p,t)-C_0(t)) \quad
\quad \forall (p,t)\in M\times \R .\end{equation}  In fact,
(\ref{e1}) would yield now $d(u+2V)/dt < C(u+2V-B)$ for some
constants $C>0, B\in \R$ which depend on the domain $[0,b),
b<\infty$. So, $u+2V$ (and, then, $u$) would be bounded as a
subsolution, see \cite{CRSago2012} for details).

\begin{svgraybox}
These new  bounds (lower for $V$ plus (\ref{eff})) are independent
of those in {\em (iii)} because, when $V$ grows fast to infinity,
such a growth is allowed for $\partial V/\partial t$ too. So, to
find a general optimal bound for $\partial V/\partial t$ (say, extending all previous
with some nice geometric interpretation) remains as a natural question.
\end{svgraybox}

\subsubsection{Notes on the general Finsler case}
\label{s3.2}

\runinhead{Finsler metrics, second order equations and strong convexity.}  In order
to extend previous results to the Finslerian setting, notice that
 the Riemannian  metric $g$ in Theorem \ref{t1} not only allows to introduce distances and
 estimates on the
 growth of tensor fields, but also becomes essential to pose the
 second-order differential equation (\ref{e}). Thus, for the Finslerian extension,  not
 only higher differentiability for the Finsler metric $F$ will be required but also
 its {\em strong convexity}, to be explained here.

\begin{remark}
As pointed out in Section \ref{s1}, the existence of smooth
Finsler metrics introduce some restrictions in the infinite
dimensional case. In fact, notions such as
pseudo-gradients\footnote{ \label{footpseudo} According to Palais \cite[Defn.
4.1]{PalaisTop_LustSch66} (and taking into account Moore's
modification \cite[p. 50]{Moore}), a pseudo-gradient for a
function $V$ on an open  subset $U$ is a locally Lipschitz vector
field $X$ such that $\epsilon^2 F_p(X_p)^2\leq  \parallel
dV_p\parallel \leq \epsilon^{-2}
dV_p(X_p)$ for all $p\in U$.  
} were introduced to avoid those restrictions. Recall that the
smoothness of each pointwise norm $F_p$ is required only away from
0 and, thus,  it can be characterized as the smoothness of the
$F_p$-unit sphere as a submanifold of the corresponding vector
space $T_pM$. However, the smoothness of $F$ is not enough to
introduce connections, covariant derivatives, etc., which appears implicitly in \eqref{e}.
\end{remark}
The triangle inequality implies that, for each norm $F_p$, $p\in
M$, the closed unit ball $\bar B_p(0,1)$ is convex, i.e., it
contains any segment with endpoints in $\bar B_p(0,1)$. If the
triangle inequality holds strictly, then the unit sphere is
strictly convex, in the sense that each segment with endpoints in
$\bar B_p(0,1)$ must be entirely contained in the open unit ball $
B_p(0,1)$ except, at most, the endpoints. Nevertheless, even in
the smooth finite-dimensional case, the unit sphere may be
strictly convex but not  strongly convex in the following sense.

Recall that the {\em fundamental tensor}   of each norm $F_p$ is
the tensor field on $T_pM\setminus \{0\}$ defined as the Hessian
$h_{v_p}$ of $F_p^2$ at each $v_p\in T_pM\setminus\{0\}$. Such a
Hessian can be defined by using the affine connection of  $T_pM$
if $F_p$ is $C^2$. Now,   consider the  slit tangent bundle $
TM\setminus\{0\}$ and the tangent bundle $TM$, as well as the
natural projection $ \pi: TM\setminus\{0\} \rightarrow M$. This
maps induces a vector bundle $\pi^*(TM)$ with base
$TM\setminus\{0\}$, being its fiber at each $v\in
TM\setminus\{0\}$ isomorphic to $T_{\pi(v)}M$. Taking the
fundamental tensor for each $F_p, p\in M$, one defines naturally
the fundamental tensor field $h$ of $F$ as a tensor field on the
vector bundle $\pi^*(TM)$, and $F$ is called {\em strongly convex} when
$h$ becomes a smooth positive definite tensor.

Strong convexity may introduce a new restriction in the
infinite-dimensional case, but it is necessary for several
purposes, even in the case of $n$-manifolds (see \cite{JS} for
details):

\begin{itemize}\item To ensure that geodesics (defined
as extremals of the energy functional) are determined univocally
by its initial condition (starting point and velocity) at some
point. That is, otherwise {\em geodesics cannot be regarded as
solutions of a second order differential equation} nor their
velocities yield integral curves on a vector field on $TM$.

\item To ensure (at least in the finite-dimensional case) that the
natural Legendre transformation  $TM\rightarrow TM^* , v_p\mapsto
g_{v_p}(v_p,\cdot )$ (which generalizes the metric isomorphism of
inner spaces, see (\ref{ebemol}), but may not be linear) becomes a
diffeomorphism. Recall that this map is the fiber derivative
associated  to the Lagrangian $L=F^2/2$  (see \cite[Sect.
3.1]{Sh}, \cite[Sect. 3.6]{AMR}) and, then, the Lagrangian becomes
hyper-regular. In this case gradients can be defined, and
pseudo-gradients (see the footnote \ref{footpseudo}) are no longer necessary.

\item To define  natural  connections on the Finsler manifold.
\end{itemize}


\runinhead{Standard Finsler case.} Taking into account the
difficulties pointed out above for the general Finsler case, we
restrict now to {\em standard Finsler manifolds} i.e.,
$n$-manifolds endowed with a $C^\infty$-smooth and strongly convex
Finsler metric. This is the object of study of standard references on Finsler manifolds as, for example\footnote{However, standard Finsler metrics are usually allowed to be non-reversible, see Remark \ref{rfuturestudy}.},  \cite{BCS, Sh}. Some similarities
with the Riemannian case appear then:

\begin{itemize}
\item A covariant derivative for vector fields on curves exists. Thus,
the acceleration of these curves can be defined, extending so the
notion of $D\dot\gamma/dt$ in the Riemannian case \cite[pp.
121-124]{BCS}, \cite[Sect. 5.3]{Sh}.

\item  Non-constant geodesics can be defined as curves with 0
acceleration, they admit a variational characterization and they
also determine a (second order equation) vector field $G$ on the
slit tangent bundle
$TM\setminus\{0\}$ 
so that the integral  curves of $G$ are the curves of velocities
of  geodesics, \cite[Sect. 3.8, 5.3]{BCS}, \cite[Sect. 5.1]{Sh}.

\item The Finsler metric $F$ provides the fundamental
tensor
as well as a natural Sasaki type metric on the slit tangent bundle
that makes $TM\setminus\{0\}$  a Riemannian manifold \cite[p.
35]{BCS}.
\end{itemize}
Of course, important differences with the Riemannian case remain,
because  Chern/ Rundt connection  in Finslerian geometry (as well as Cartan, Hashiguchi or
Berwald  connections) becomes much subtler
than the natural Levi-Civita connection in the Riemannian case.

Bearing in mind these subtleties, one can try to give different
Finslerian extensions of Theorem \ref{t1}. Here, we will consider
just the most obvious one, and leave the possibility of obtaining 
more general results for further developments. To avoid working
with Finslerian machinery and work with one of the possible
connections, notice that, in the case $R=E=0$, formula (\ref{e})
is the Euler-Lagrange equation for the critical curves of the
action:
\begin{equation}\label{eaction} \int_{a}^b \left(\frac{1}{2}F(\dot
\gamma(t))^2-V(\gamma(t),t)\right) dt
\end{equation}
with fixed points $\gamma(a), \gamma(b)$. In Theorem \ref{t1}, $F$
is the  norm of the Riemannian metric  but, obviously, functional
(\ref{e}) makes sense for any Finsler metric and, under some  the
conditions as above, its Euler-Lagrange equation can be written as
in (\ref{e}). We say that $\gamma: I\subset \R
\rightarrow M$ is a {\em trajectory for the potential $V$} if its restriction to any compact subinterval $[a,b]$ of $I$ is a critical point of the action functional (\ref{eaction}).

\begin{proposition}\label{pFinslerExtension}
Let $(M,F)$ be a  standard Finsler manifold, and consider a $C^1$
time-dependent potential $V: M\times \R\rightarrow \R$ such that
$-V$ and $|\partial V/\partial t|$ grows at most quadratically for
finite times. Then, any inextensible trajectory $\gamma: I\subset \R
\rightarrow M$
for the potential $V$ is
complete.
\end{proposition}

\begin{proof}
Notice first that the problem can be reduced to study the integral
curves of a  vector field on $TM$, because, as the Finsler metric
is standard, the Lagrangian $L=(F^2/2)-V$ becomes regular (in
fact, hyper-regular), see for example  \cite[Th. 3.5.17,
3.8.3]{AMR}. Then, putting $u=F(\dot\gamma)^2$, one has
$d(u+2V)/dt=2\partial V/\partial t$ and formula (\ref{e2}) holds
(with $A_1=0$), so that the proof follows as in Theorem \ref{t1}.
\end{proof}

\begin{remark}\label{rfuturestudy}
A different direction in the possible  generalizations of Theorem
\ref{t1}, is to allow non-reversible Finsler metrics,
so that $F(v)\neq F(-v)$ in general. This leads us to consider {\em
generalized distances} (i.e., possibly non-symmetric ones) and
then, {\em forward and backward} geodesics and Cauchy completions,
as well as many other subtleties (see \cite{FHS} and references
therein). Nevertheless, the general background for completeness
would be maintained for this case. In fact, Proposition
\ref{pFinslerExtension} can be extended to the non-reversible
case. Namely,   regarding the hypotheses of completeness for $F$
in the sense of, say, {\em forward completeness}, and the
generalized distance $d_F$ to the base point in the ordering
$|p|=d_F(p_0,p)$ (so that the
 bound for the potential remains formally equal). Then, the  technique works
 also for the non-reversible case, and
 the conclusion of {\em forward} completeness
 still holds.
\end{remark}

\section{Completeness of pseudo-Riemannian geodesics}
\label{s4}

This section is divided into four parts. The fist  tries to
orientate the intuition about completeness on indefinite manifolds by
recalling some examples. Moreover,  the role of incompleteness
in relativistic singularity theorems is compared with the role of finite
 diameter
for  some Riemannian Myer's-type results. In the
second part, we recall some results on completeness for manifolds
with a high degree of symmetry. Here, the difference
between global symmetries (homogeneous, symmetric spaces) as in
Theorems \ref{tMa}, \ref{tRS}
 and local ones (constant curvature, local symmetry) in
 Theorems
\ref{tLa}, \ref{tCaKl} becomes apparent. The third part is focused on plane wave
type spacetimes, whose completeness yields a direct link with the
Riemannian results of trajectories under potentials (Theorem
\ref{twaves}). Previous results suggest some open questions stated
in the last part of the section.

In what follows, $(M,g)$ will be a $n$-manifold endowed with a
pseudo-Riemannian metric of index $\nu$, typically a Lorentzian
one (i.e., $\nu=1$ so that the  signature is $(-,+,\dots , +)$).
The name of {\em semi-Riemannian} manifold (instead of
pseudo-Riemannian) has been also spread, especially since
O'Neill's book \cite{ON}. This book is referred to here  for general
background on pseudo-Riemannian geometry,  the review \cite{CS}
for the specific problem of geodesic completeness, and the
book \cite{BEE} for
 related Lorentzian results.

\subsection{The pseudo-Riemannian and Lorentzian settings}
Let $(M,g)$ be a pseudo-Riemannian manifold, and $v\in TM$, $v\neq
0$. Extending the nomenclature in General Relativity, $v$ will be
called {\em timelike} (resp. lightlike, spacelike) if $g(v,v)<0$
(resp. $=0$, $>0$).

\runinhead{Abandoning the Riemannian intuition.} For a
pseudo-Riemannian manifold there is no any result analogous to the
Hopf-Rinow one and, for example, $M$ may be compact and
geodesically incomplete.

\begin{example}\label{exnohopfrinow} Consider the Lorentzian metric $g$ on $\R^2$
defined as $g=2dxdy +\tau(x)dy^2$, where $\tau$ is periodic of
period 1, $\tau(0)=0$ and $ \tau'(0)\neq 0$. A simple computation
shows that the line $x=0$ can be reparameterized as an incomplete
lightlike geodesic. So, the quotient torus $T=\R^2/\Z^2$ inherits
an incomplete Lorentzian metric (more refined properties on tori
can be found in \cite{Strans} and references therein).
\end{example}

The previous example also shows that a closed lightlike geodesic
may be non-periodic and, then, incomplete. Also as a difference
with the Riemannian case, a homogeneous Lorentzian manifold may be
incomplete.

\begin{example}\label{exnohomogcompl} Consider a half plane of Lorentz-Minkowski space in lightlike coordinates $u,v$
namely $(\R^+\times \R, g=2dudv)$. This space is trivially
incomplete, and it is homogeneous too, as both, the
$v$-translations and the maps $\Phi_\lambda:(u.v)\mapsto (\lambda
u, v/\lambda)$ (for any $\lambda>0$), are isometries. Recall also
that the quotient cylinder obtained from the orbits of the
isometry group $\{\Phi_2^m : m\in \Z\}$ is another example of
space with  a closed incomplete lightlike geodesic (namely, the
projection of $u\mapsto (u,0)$).

\end{example}
\runinhead{Singularity theorems.} Even though at the very
beginning  of General Relativity incompleteness was regarded as a
pathological property for a physical spacetime, the further
development of Relativity showed that incompleteness appears
commonly under physically realistic conditions. Well-known results in this
direction were obtained by Raychaudhuri  \cite{Ra}, Penrose
\cite{Pe}, Hawking \cite{Ha}, Gannon \cite{Ga} or, more recently,
Galloway and Senovilla \cite{GS}, amongst others (see
the review \cite{Seno} for general background). We emphasize that the claimed
incompleteness here occurs only for geodesics of timelike or
lightlike type\footnote{Explicit examples by Kundt \cite{Kundt},
Geroch \cite[p. 531]{Ge} and Beem \cite{Beem} showed the full
logical independence among spacelike, timelike and lightlike
geodesic completeness.}. Even though it is not totally clear to
what extent such incomplete geodesics would represent a physical
singularity (as well as the meaning of the latter, see the
classical discussion \cite{Ge}), the moral in Relativity is that
the knowledge of  the possible completeness or incompleteness of
the underlying Lorentzian manifold becomes an essential property
of the spacetime.

As pointed out in \cite{SaBilbao}, perhaps the simplest
singularity theorem for researchers interested in connections with
Riemannian Geometry is the following one by Hawking, which can be
regarded as a support for the physical existence of a {\em Big
Bang}.

\begin{theorem}\label{thawking}
Let $(M,g)$ be a spacetime  satisfying the following conditions:
\begin{enumerate} 
\item $(M,g)$ is globally hyperbolic, \item there exists some
spacelike Cauchy hypersurface $S$ with an infimum $C>0$ of its
expansion, that is, such that its mean curvature vector
$\vec{H}=H\vec{n}$, where $\vec{n}$ is the future-directed unit
normal, satisfies
 $H\geq C>0$,
 \item the  timelike convergence condition holds: ${\rm Ric}(v,v)\geq 0$ for any timelike vector $v$.
\end{enumerate}
Then, any past-directed timelike curve starting at $S$ has length
at most $1/C$.
\end{theorem}
The reason is that the proof of this theorem can be regarded as
isomorphic to the proof of the following purely Riemannian result:
\begin{theorem}\label{tRHawking}
Let $(M,g)$ be a  Riemannian manifold satisfying:
\begin{enumerate}
\item $g$ is complete, \item there exists some embedded
hypersurface $S$ which separates $M$ as a disjoint union
$M=M_-\cup S\cup M_+$, with an infimum $C>0$ of its expansion
towards $M_+$, that is, such that its mean curvature vector
$\vec{H}=H\vec{n}$, where $\vec{n}$ is the  unit normal which
points out $M_-$, satisfies
 $H\geq C>0$,
 \item Ric$(v,v)\geq 0$ for every $v$.
\end{enumerate}
Then, ${\rm dist}(p,S)\leq 1/C$ for every $p\in M_-$.
\end{theorem}
In fact, this last theorem can be proven by using standard
techniques on focal points and Myers' theorem. Such techniques can
be extended to the Lorentzian setting by realizing that the roles
of each one of the three hypotheses in Theorem \ref{thawking} is
isomorphic in the proof to the corresponding hypothesis in Theorem
\ref{tRHawking} (in particular, the role of Riemannian
completeness is played by global hyperbolicity), see
\cite{JSlibro} for full details. The techniques of singularity
theorems, however, become much more refined, because of the
weakening of causality assumptions,  the appearance of genuinely
Lorentzian elements such as trapped surfaces and other subtleties, see for example
\cite{HP} or, more recently,  \cite{GS}.

\subsection{Completeness under symmetries}

After previous considerations, it is clear that some strong
assumptions will be required in order to prove geodesic
completeness. We will focus on some types of symmetries.

\runinhead{Killing and conformal fields.} The simple Examples
\ref{exnohopfrinow}, \ref{exnohomogcompl} of non-complete compact
or homogeneous Lorentzian manifolds, make apparent the importance
of the following theorem by Marsden  \cite{MaHomog} (see also
\cite[4.2.22]{AM}):

\begin{theorem} \label{tMa} {\em \cite{MaHomog}}
Any compact homogeneous pseudo-Riemannian manifold is geodesically
complete. 
\end{theorem}
Marsden's proof is carried out by proving  that $TM$ can be
written as the union of compact subsets $S_\alpha$, each one
invariant by the geodesic flow (and, so, Proposition \ref{pextend}
yields directly the result). In fact, if $\mathfrak{g}^*$ is the
dual of the Lie algebra of the isometry group, and $P:
TM\rightarrow \mathfrak{g}^*$ is the momentum map (i.e.,
$P(v)\xi=g(v,\xi_M)$, where $\xi_M$ is the infinitesimal generator
of $\xi\in \mathfrak{g}$), then $S_\alpha=P^{-1}(\alpha)$, for
each $\alpha\in \mathfrak{g}^*$.

As proven by Romero and the author \cite{RSpams, RSgeoded}, this
result can be extended in two directions. Firstly,  it is not
necessary, in order to ensure the completeness of each geodesic
$\gamma$, that its velocity $\dot\gamma$  remains in a compact
subset of $TM$. In the spirit of Proposition
\ref{pcompleteboundedsubintervals},  it is enough if it remains in
a compact subset when its domain is restricted to bounded
intervals. From such an observation, Theorem \ref{tMa} can be
extended to metrics conformal to Marsden's. Secondly, a
homogeneous manifold is full of Killing vector fields but if,
say, a compact Lorentzian manifold admitted just one {\em
timelike} Killing vector field\footnote{This case is interesting
also for the classification of  flat compact Lorentzian manifolds,
which are called then {\em standard}, see \cite{Ka}.} $K$, this
would be enough. Indeed, as $g(\dot\gamma,K)$ is a constant for
any geodesic $\gamma$, this (plus the constancy of $g(\dot
\gamma,\dot\gamma)$) is sufficient to ensure that $\dot\gamma$
lies in a compact subset. So, from these ideas:

\begin{theorem} {\em \cite{RSpams, RSgeoded}}
A compact pseudo-Riemannian manifold $(M,g)$ of index $\nu$ is
geodesically complete if one of the following properties hold:

\bit \item $(M,g)$ is (globally) conformal to a homogeneous one,
or \item $(M,g)$ admits $\nu$ timelike conformal vector fields which are
pointwise independent. \eit \label{tRS}
\end{theorem}
The technique  also admits extensions to non-compact manifolds,
see \cite{RSgeoded}, \cite{SaGRG}; for applications to
classification of spaceforms, see \cite{Ka}. Further results on
locally homogeneous 3-spaces (involving also the classification of
these spaces) can be found in \cite{BlMe},
\cite{Cal} and \cite{DZ}.

 \runinhead{Locally symmetric and constant curvature
manifolds.} As a difference with homogeneous spaces, it is easy to
check that any {\em pseudo-Riemannian symmetric space is
geodesically complete} (see for example \cite[Lemma 8.20]{ON}).
Nevertheless, even for locally symmetric spaces and, in
particular, constant curvature ones, the problem is not as trivial
as it may seem. We quote two results which will be relevant in
order to state some open questions below. The first one is due to
Lafuente:
\begin{theorem} {\em \cite{La}}
For a locally symmetric Lorentzian manifold, the three types of
causal completeness (timelike, lightlike and spacelike) are
equivalent. \label{tLa}
\end{theorem}
The second one was proven by Carri\'ere \cite{Ca} in the  flat
case and extended by Klinger \cite{Kl} for  manifolds of any
constant curvature.

\begin{theorem} {\em \cite{Ca, Kl}}
Any compact Lorentzian manifold of constant curvature is
geodesically complete. \label{tCaKl}
\end{theorem}

\begin{remark}\label{ropenqconstantcurvatarbsign} Recall that the proof of this result holds only for
Lorentzian signature; as far as we know, the extension of the
result to higher signatures is an open problem.
\end{remark}

\subsection{Riemannian and Lorentzian interplay: plane waves}

\runinhead{Plane waves, pp-waves and further generalizations.}
Following  \cite{CFS}, consider a Lorentzian $n$-manifold, $n\geq 3$, that can be written
globally as $(M=\R^2\times M_0,g)$ where the natural coordinates
of $\R^2$ will be labelled $(u,v)$ and $g$ is written as:
$$
g_{(u,v,x)}=-2dudv + H(u,x)du^2 + \Pi_0^\star g_0, \quad \quad
\forall (u,v,x)\in \R^2\times M_0,
$$
being $\Pi_0:M\rightarrow M_0$ the natural projection and $g_0$ a
Riemannian metric on $M_0$. Here, we will refer to
these spaces as {\em $M_0$p-waves}. When $(M_0,g_0)$ is just
$\R^{n-2}$, these metrics are called {\em pp-waves} ({\em
plane-fronted waves with parallel rays}), namely, $M=\R^n$,
\begin{equation}
\label{epp}
g_{(u,v,x)}=-2dudv + H(u,x^1,\dots,x^{n-2})du^2 + \sum_{i=1}^{n-2}
(dx^i)^2 \quad \quad \forall (u,v,x^1,\dots,x^{n-2})\in \R^n
\end{equation}
Such a pp-wave is called a {\em plane wave} when $H$ is quadratic
in $(x^1,\dots , x^{n-2})$,
$$
H(u,x^1,\dots,x^{n-2})= \sum_{i,j=1}^{n-2}A_{ij}(u)x^ix^j.
$$
 In the particular case $n=4$ one writes
$H(u,x,y)=a(u) \, (x^2-y^2) + 2 \, b(u) \, xy + c(u) \, (x^2+y^2),
$ where $a,b,c$ are arbitrary smooth functions of $u$. The
functions $ a,b$ describe the wave profiles of the two linearly
independent polarization modes of gravitational radiation, while
$c$ describes the wave profile of  non-gravitational radiation.
When $c = 0$ (vacuum or gravitational plane waves) the Ricci
tensor vanishes.

 Plane waves
are interesting  in many physical issues. We remark here that they
are also interesting in  the framework of {\em $r$th-symmetric
spaces} $r\geq 2$ (introduced in \cite{Se}, see \cite{BSS}  for a
systematic study). These are pseudo-Riemannian manifolds with
$r$th-covariant derivative of its curvature tensor $R$ equal to 0:
$$\nabla^{r}R:=\nabla \dots ^{(r)}\nabla R\equiv 0.$$ For
Riemannian manifolds $r$th-symmetry implies local symmetry (i.e.,
$\nabla R=0$) but proper examples of rth-symmetric spaces  can be
found in the class of plane waves. In fact, such examples are
obtained just regarding the matrix $A$ as a polynomial in $u$ of
degree $r-1$:
$$A_{ij}(u)=a_{ij}^{(r-1)} u^{r-1}+\ldots+a_{ij}^{(1)}
u^{1}+a_{ij}^{(0)}$$ where $a_{ij}^{(r-1)}\not\equiv 0$; a simple
computation shows that  $\nabla^rR=0$ but $\nabla^{r-1}R\neq 0$.

As shown in \cite{BSS},  proper 2nd-symmetric Lorentzian spaces
are locally isometric to the product of such a wave (with $r=2$)
and a locally symmetric Riemannian space.

\runinhead{Completeness of $M_0$p-waves.} A nice relation between
the geodesic completeness of a class of Lorentzian manifold and
the completeness of Riemannian trajectories for a potential
appears in the case of $M_0$p-waves:

\begin{theorem} \label{twaves}A $M_0$p-wave is geodesically complete if and only if $(M_0,g_0)$ is complete
 and the trajectories of
$$\frac{D\dot\gamma}{dt}(t)\ =   - \nabla^{M_0} V (\gamma(t),t)$$
are complete for $V=-H/2$.

Thus, under the completeness of the Riemannian part $(M_0,g_0)$, a
$M_0p$-wave is complete if $H$ and $|\partial H/\partial u|$ grows
at most quadratically for finite  $u$-times. In particular, all
plane waves are geodesically complete. \end{theorem}

\begin{proof} The first part is proven  in \cite[Th. 3.2]{CFS}, by
means of a careful equivalence between the Lorentzian geodesics
and Riemannian trajectories \cite[Prop. 3.1]{CFS}. So, it is
enough to apply Theorem \ref{t1}.
\end{proof}
As emphasized in \cite{FScqg}, this type of  result also justifies
that all physically reasonable pp-waves (that is, those with a
qualitative behavior of $H$ as a plane wave, eventually with some
possible decay at infinity) will be geodesically complete and, so,
they can be regarded as singularity free.

Finally, we state the following very recent result by Leistner and
Schliebner on pp-waves. Notice that, for any pp-wave as  above
(formula \eqref{epp}), the vector field $V=\partial_v$ is parallel
and lightlike, and the curvature tensor $R$ satisfies:
\begin{equation}
\label{elocalpp}
R(U,W) = 0 \qquad\qquad
\hbox{for all vector fields} \; U,W \; \hbox{orthogonal to} \; V .
\end{equation}
 Conversely, any spacetime admitting such a vector field can be written locally as a pp-wave. Now:

\begin{theorem}\label{tleistner} {\em \cite{LS}}
The universal covering of any compact Lorentzian manifold $(M,g)$ which admits a parallel lightlike vector field $V$ satisfying \eqref{elocalpp}, is a geodesically complete pp-wave.
In particular, $(M,g)$ is geodesically complete.
\end{theorem}
 This result  goes in the same direction of those for the compact case with constant curvature or (conformal) homogeneity.
 Nevertheless,  the  special holonomy derived from the
 {\em global} existence of $V$ plays a fundamental role here.
 So, in principle, it is not enough to assume that the spacetime is just locally isometric to a plane wave
 (and, so, for example, Carri\'ere's theorem is not re-proved).
 In fact, the universal covering is taken such that $\partial_v$ is the lift of the globally defined vector field $V$.

\begin{remark} The existence of a  complete vector field $V$ fulfilling
 the hypotheses in Theorem \ref{tleistner}, can be also regarded as
 a generalization of the notion of pp-wave to non-trivial
 topology. As emphasized in \cite{LS}, such a generalization may
 pose some topological subtleties related to Ehlers-Kundt conjecture (see the third question below),
 loosely suggested in the original article \cite{EK}.
\end{remark}

\subsection{Some open questions}

Taking into account previous considerations, the following
questions become natural and are open, as far as we know:

\ben \item \begin{svgraybox} Assume that a compact Lorentzian
manifold is globally conformal to a manifold of constant
curvature. Must it be geodesically complete? \end{svgraybox}

Recall that this poses a possible extension of Theorem
\ref{tCaKl}, which may be expected after the conformal extension
in Theorem \ref{tRS} of Marsden's Theorem \ref{tMa}. It is also
worth pointing out that, for compact manifolds, {\em lightlike}
completeness is a conformal invariant (this is easy to check as
lightlike pregeodesics are conformally invariant, and their
reparameterizations  as geodesics depend on a bounded conformal
factor, see \cite[Section 2.3]{CS} for detailed computations). So,
if a counterexample to the question existed, it would be
incomplete in some causal sense and complete in the lightlike
case. In particular, such a counterexample would prove that Lafuente's Theorem
\ref{tLa} cannot be extended to the conformal case even for
compact manifolds. It is also worth pointing out that, if $\gamma$ is a
geodesic for a metric $g$, then it satisfies an equation type \eqref{e}
for any conformally related metric $\bar g$
(but, in this case, such an equation is posed on a Lorentzian manifold $(M,\bar g)$).

\item
\begin{svgraybox}Assume that a pseudo-Riemannian manifold is r-th symmetric. Must
the three types of causal completeness be equivalent?
\end{svgraybox}

Such a question becomes natural after Lafuente's Theorem
\ref{tLa}, especially in the case of Lorentzian 2nd-symmetric
spaces, because of their simple classification explained above.

\item  \begin{svgraybox} Must any complete gravitational (i.e.,
Ricci flat)  pp-wave be a plane wave?
\end{svgraybox}
This is a long-standing open problem posed by Ehlers and Kundt
\cite{EK}. Recall first that all plane waves are complete, even if
non-gravitational (Theorem \ref{twaves}). The fact that these
waves are gravitational, i.e., Ricci flat, yields a link with
complex variable, as this condition is equivalent to the
harmonicity  of $H(x,u)$ with respect to the variable $x$ (see
\cite{FS}) ---notice that the study of the completeness of holomorphic vector
fields, become a field of research in its own right which has been
handled with specific tools, see for example \cite{Nueva,
Brunella, Bustinduy}. Thus, there are both,  physical and
mathematical motivations for its study \cite{Bi, FS}.

 \een As a
last comment, we point out that the completeness of trajectories
in a Lorentzian manifold under external forces is an almost open
field with rich possibilities \cite{CRSburgos}; as we have
said in the comments to question 1, this includes the equation of
geodesics for a conformally related metric. So, even though
 the physical interpretations of
such forces are less apparent  in the Lorentzian case than in the
Riemannian one, this may be an interesting topic for future
research.

\end{document}